\documentclass[preprint, review]{elsarticle}
\usepackage{amsfonts}
\usepackage{amsmath}
\usepackage{amsthm}
\usepackage{amssymb}
\usepackage{graphicx}
\usepackage{amsfonts}
\usepackage{tikz}
\usetikzlibrary{arrows,matrix,positioning}
\usepackage{stmaryrd}

\usepackage{caption}
\usepackage{subcaption}

\usepackage{array}

\usepackage{amssymb}

\DeclareMathOperator{\character}{char}
\DeclareMathOperator{\stack}{ \bowtie }

\newcommand{\weave}{\mathbin{\rotatebox[origin=c]{180}{$\merge$}}}

\newtheorem{thm}{Theorem}
\newtheorem*{thm*}{Theorem}
\newtheorem{lem}[thm]{Lemma}

\newtheorem{coro}[thm]{Corollary}
\newtheorem{defn}{Definition}

\theoremstyle{remark}

\journal{Mathematical Social Sciences}

\begin{document}

\begin{frontmatter}



\title{Cubic Preferences and the Character Admissibility Problem\tnoteref{t1}}
\tnotetext[t1]{This research was supported by the National Science Foundation under Grant No.\ DMS-1262342, which funded a Research Experiences for Undergraduates (REU) program at Grand Valley State University. Any opinions, findings and conclusions or recommendations expressed in this material are those of the author(s) and do not necessarily reflect the views of the National Science Foundation (NSF). Declarations of interest: none}


\author[bb]{Beth Bjorkman} \ead{bjorkman@iastate.edu}
\author[sg]{Sean Gravelle}\ead{sean.gravelle@huskers.unl.edu}
\author[jkh]{Jonathan K. Hodge\corref{cor1}}\ead{hodgejo@gvsu.edu}

\address[bb]{Iowa State University}
\address[sg]{University of Nebraska-Lincoln}
\address[jkh]{Grand Valley State University}

\cortext[cor1]{Corresponding author; mailing address: Department of Mathematics, Grand Valley State University, 1 Campus Drive, Allendale, MI 49401}

\begin{abstract}
In multiple-question referendum elections, the separability problem occurs when a voter's preferences on some questions or proposals depend on the predicted outcomes of others. The notion of separability formalizes the study of interdependence in multidimensional preferences, and the character admissibility problem deals with the construction of voter preferences with given separability structures. In this paper, we develop a graph theoretic approach to the character admissibilty problem, using Hamiltonian paths to generate voter preferences. We apply this method specifically to the hypercube graph, defining the class of cubic preferences. We then explore how the algebraic structure of the group of symmetries of the hypercube impacts the separability structures exhibited by cubic preferences. We prove that the characters of cubic preferences satisfy set theoretic properties distinct from those produced by previous methods, and we define two functions to construct cubic preferences. Our results have potential applications to experimental work involving election simulation.
\end{abstract}

\begin{keyword}



referendum elections \sep separability \sep binary preference matrices \sep characters \sep admissibility \sep hypercube \sep Gray code \sep Hamiltonian paths

\end{keyword}

\end{frontmatter}


\section{Introduction}\label{intro}

In referendum elections, voters are often required to cast simultaneous ballots on multiple questions or proposals. The \emph{separability problem} \cite{sepproblem} occurs when voters' preferences on some sets of questions depend on the outcomes of others. For example, a voter may support a property tax increase, but only if a proposal to improve roads is also approved. In a simultaneous election, voters have no way to express such interdependencies, which can lead to election outcomes that are unsatisfactory or even paradoxical. One classic example is due to Lacy and Niou \cite{lacyniou}, who demonstrate the possibility of an election in which the winning combination of outcomes is the last choice of \emph{every} voter.

The notion of \emph{separability} formalizes the study of interdependence in multidimensional preferences and is important in many fields,  including economics, social choice theory, operations research, and computer science. Here we focus on the structure of interedependent preferences in multiple-criteria binary decision processes such as referendum elections. In this context, a voter's preferences on a set of questions are said to be \emph{separable} if they do not depend on the outcome of other questions in the election (and \emph{nonseparable} otherwise).

Given a convenient representation of a voter's preferences (we use \emph{preference matrices}), it is relatively straightforward to determine the sets of questions that are separable with respect to that voter. The collection of all such sets is called the voter's \emph{character}. However, the corresponding inverse problem is more complex. In fact, there are no general methods for constructing a preference order with a given character, and it is sometimes impossible to determine whether such a preference order even exists. If such an order does exist, the associated character is said to be \emph{admissible}. Past research has focused on classifying and constructing admissible characters. Bradley, Hodge, and Kilgour \cite{bradley05} proved results implying that admissible characters are always closed under intersections but may not be closed under other set operations. Hodge and TerHaar \cite{hodgeterhaar08} determined all admissible characters for question sets of size 4 or less; they also proved the existence of nontrivial inadmissible characters---that is, collections of sets that are closed under intersections and yet cannot occur as a voter's character---for question sets of size 4 or more. Most recently, Hodge, Krines, and Lahr \cite{hodgekrineslahr09} developed the technique of \emph{preseparable extensions} to construct preferences with certain classes of characters.

In this paper, we introduce a new method that uses vertex-edge graphs to generate preference orders. In Section 2, we introduce this method and apply it the $n$-dimensional hypercube graph, defining a class of preferences, called \emph{cubic preferences}. In Section 3, we prove results that relate the separability structures associated with cubic preferences to the algebraic structure of the hypercube graph. In Section 4, we present methods for constructing cubic preferences and prove results about the resulting characters. Importantly, we show that the characters associated with cubic preferences are distinct from those generated by Hodge, Krines, and Lahr's method of preseparable extensions. We conclude in Section 5 with a brief discussion of the potential applications of our methods to election simulation and the character admissibility problem.


\section{Definitions}

\subsection{Preference Matrices and Separability}

Throughout this paper, we will limit our attention to decision-making processes involving a finite number of binary (yes-no) decisions, using the language of referendum elections to ground our work in a familiar context. In a referendum election with $n$ yes-no questions, there are $2^n$ possible results for the election as a whole, which we refer to as \emph{outcomes}.  If we let $Q_n = \{1, 2, \ldots, n\}$ denote the \emph{question set},
then an outcome is simply an element of $\{0, 1\}^n$ (the Cartesian product of $n$ copies of $\{0, 1\}$), which we denote by $X$. For each question, we associate a value of $1$ with an outcome of \emph{yes} and $0$ with an outcome of \emph{no}. Given a subset $S \subseteq Q_n$ with $|S| = m$, we can define an outcome \emph{on} $S$ in a similar manner---namely, as an element of $\{0, 1\}^m$. We use the notation $X_S$ to represent the set of all possible outcomes on $S$.  


Consistent with previous research, we use a total order $\succ$ on $X$, displayed in convenient form as a \emph{preference matrix}, to represent the preferences of each voter. In particular, for each voter we list the $2^n$ possible outcomes as the rows of a $2^n \times n$ matrix, with the voter's most preferred outcome as the first row and the voter's least preferred outcome as the last row.

As an example, consider the following preference matrix for an election on $n = 2$ questions:\footnote{The examples involving matrix $A$, here and below, first appeared in \cite{hodgeterhaar08}.}
\[ A=
\begin{pmatrix}
1&0\\
1&1\\
0&1\\
0&0\\
\end{pmatrix}
\]

This matrix represents the preferences of a voter whose most preferred outcome is \emph{yes} on the first question and \emph{no} on the second, and whose least preferred outcome is \emph{no} on both questions. Written horizontally, the corresponding total order is \[ (1,0) \succ (1,1) \succ (0,1) \succ (0,0), \] or simply \[ 10 \succ 11 \succ 01 \succ 00. \] (We will typically omit tuple notation except in cases where it is necessary for clarity.)

A subset $S \subset Q_n$ is said to be \emph{separable} with respect to a given voter if that voter's preferences on questions in $S$ do not depend on the outcome of questions outside of $S$. In other words, we can determine the separability of the set $S$ by successively fixing all possible outcomes on $Q_n - S$ and checking to see if the preferences induced on $S$ are the same for each of these outcomes.

To illustrate, consider the matrix $A$ again, and suppose we restrict our attention to the outcomes in which the result on the second question is $0$. These outcomes correspond to the top and bottom rows of the matrix, inducing an order of $1 \succ 0$ on the first question. Likewise, the middle two rows represent outcomes in which the result on the second question is $1$, and these two rows again induce an order of $1 \succ 0$ on the first question. The fact that these induced orders are the same---regardless of what result is fixed on the second question---indicates that the first question (or, more precisely, the set $\{1\}$) is separable.

In contrast, fixing an result of $0$ on the first question induces an order of $1 \succ 0$ on the second question (as shown by the bottom two rows of $A$), whereas fixing a result of $1$ on the first question induces an order of $0 \succ 1$ on the second question (as shown by the top two rows of $A$). Because these two induced orderings are different, the second question (i.e., the set $\{2\}$) is \emph{not} separable. Intuitively, this voter's preference on the outcome of the second question depends on what happens on the first question. If the first question passes, the voter wants the second to fail. And if the first question fails, the voter wants the second to pass.

For larger question sets, we can consider not only the separability of individual questions but also of sets of questions. Analogous to the previous example, one way to test the separability of a set $S$ is to consider the submatrices induced by fixing various outcomes on $Q_n -S$. To illustrate,
consider the following preference matrix, $B$:


\[B=\begin{pmatrix} 
1 & 0 &1 \\ 
1 & 0 & 0 \\ 
1 & 1 & 0 \\ 
0 & 0 & 0 \\ 
0 & 1 &  1 \\ 
1 & 1 & 1 \\ 
0 & 0 & 1 \\ 
0 & 1 & 0 
\end{pmatrix}.\]

We will begin, as an example, by determining whether or not the set $\{1,2\}$ is separable. In order to do so, we must compare the induced preference orders on $\{1,2\}$ when the outcomes of $1$ and $0$, respectively, are fixed on the set $\{3\}$. It is convenient to represent these induced preference orders as follows:

\[
B^{[\{3\},1]}= \begin{pmatrix} 1 & 0 \\ 0 & 1 \\ 1 & 1 \\ 0 & 0 \end{pmatrix}, B^{[\{3\},0]}= \begin{pmatrix} 1 & 0 \\ 1 & 1 \\ 0 & 0 \\ 0 & 1 \end{pmatrix}
\]

We see that $B^{[\{3\},1]}$ represents the preference order induced on $\{1,2\}$ when the outcome 1 is fixed on $\{3\}$, and $B^{[\{3\},0]}$ represents the same, but with the outcome 0 fixed $\{3\}$. Since $B^{[\{3\},1]} \neq B^{[\{3\},0]}$, this voter's preferred ordering of the outcomes on the set $\{1, 2\}$ depends on the outcome of question 3, so $\{1, 2\}$ is nonseparable.

This example motivates our general definition of separability. In particular, if $P$ is a preference matrix on a question set $Q_n$, and $S$ is a nonempty, proper subset of $Q_n$, we let $P^{[Q_n-S,x]}$ denote the submatrix formed by fixing the outcome $x$ on $Q_n - S$ (that is, by taking the columns of $P$ that correspond to $S$ and the rows of $P$ that have an outcome of $x$ on $Q_n-S$). Note that $P^{[Q_n-S,x]}$ can itself be viewed as a preference matrix for an election on $S$. We can then define the separability of $S$ as follows: 

\begin{defn}\label{def-sep} Let $P$ be a preference matrix for the question set $Q_n$, and let $S$ be a nonempty, proper subset of $Q_n$. Then $S$ is said to be \textbf{separable} with respect to $P$ if \[ P^{[Q_n - S,x]}=P^{[Q_n - S,y]}  \text{~for all~}  x, y \in X_{Q_n - S}, \] and \textbf{nonseparable} otherwise. 
\end{defn}

We consider $Q_n$ and $\emptyset$ to be trivially separable with respect to any preference matrix. Applying Definition \ref{def-sep} to all possible nontrivial subsets of $Q_n$ yields the \emph{character} of the matrix, defined formally below.

%
%
%
%
%
%

\begin{defn}\label{def-char}
Let $P$ be a preference matrix for the question set $Q_n$. The \textbf{character} of $P$, denoted $\character(P)$, is the set of all subsets of $Q_n$ that are separable with respect to $P$. If $\character(P) = \mathcal{P}(Q_n)$ (the power set of $Q_n$), then $P$ is said to be \textbf{completely separable}. If $\character(P) = \{\emptyset, Q_n\}$, then $P$ is said to be \textbf{completely nonseparable}.
\end{defn}

For our previous examples, it is straightforward to verify that $\character(A) = \{\emptyset, \{1\}, \{1, 2\}\}$ and $\character(B) = \{\emptyset, \{1,2,3\}\}$. Thus, the matrix $B$ is completely nonseparable.

Observe that for any $S\in\character(P)$, the definition of separability induces a fixed ordering on $X_S$. As before, we use the $\succ$ symbol to denote this induced ordering (for example $1 \succ 0$), relying on context to indicate that the ordering is on $X_S$ and not $X$.

%

In addition to considering the characters of specific preference matrices, it will be convenient for us to use the term \emph{character} to refer to any collection of subsets of $Q_n$. A character $C$ is then said to be \emph{admissible} if there exists a preference matrix $P$ such that $\character(P) = C$. The fact that inadmissible characters exist, and that it is not known how to identify them based on set-theoretic properties of the characters alone, is an interesting problem in itself \cite{hodgeterhaar08}. Our approach to the character admissibility problem centers on constructing preference matrices from Hamiltonian paths in graphs.

%
%

\subsection{A Graph Theoretic Model}

Recall that, for an election on $n$ yes-or-no questions, there are $2^n$ possible outcomes. Therefore, any graph with $2^n$ vertices can be labeled with these outcomes. Moreover, any Hamiltonian path in such a graph---that is, a path that includes each vertex exactly once---traverses the (labeled) vertices in an order that can be thought of as \emph{generating} a preference matrix. The next definition formalizes this idea. 

\begin{defn}\label{def-gen}
Let $G$ be a graph on $2^n$ vertices, each labeled with a different outcome for an election on $n$ questions, and let $H$ be a Hamiltonian path on $G$. Let $P$ be the preference matrix whose rows, from top to bottom, are the labels of the vertices of $G$, in the order traversed by $H$. Then $H$ is said to \textbf{generate} $P$.
\end{defn}

For example, the Hamiltonian path shown in Figure \ref{fig-cube-path} generates preference matrix $M$.

\begin{center}
\begin{figure}[h!]
\begin{subfigure}{0.45\textwidth}\centering
\includegraphics[scale=.75]{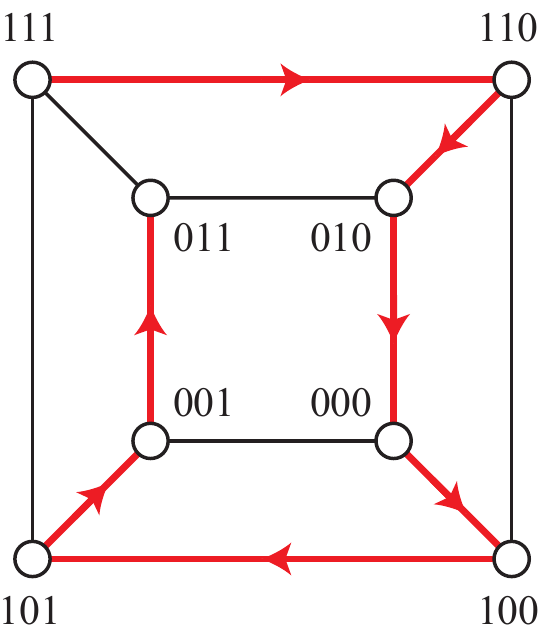}
\end{subfigure}
\begin{subfigure}{0.45\textwidth}\centering
\[M=\begin{pmatrix} 1 & 1 &1 \\ 1 &1 & 0 \\ 0 & 1 & 0 \\ 0 & 0 & 0 \\ 1 & 0 &  0 \\ 1 & 0 & 1 \\ 0 & 0 & 1 \\ 0 & 1 & 1 \end{pmatrix}\]
\end{subfigure}
\caption{Generating a preference order with a Hamiltonian path}\label{fig-cube-path}
\end{figure}
\end{center}

It is clear that, for any given graph on $2^n$ vertices, each Hamiltonian path will generate a unique preference matrix. However, graphs often contain multiple Hamiltonian paths, thereby generating multiple preference matrices. This in turn motivates the following definitions:

\begin{defn}
Given a preference matrix $P$ for an election on $n$ questions and a labeled graph $G$ with $2^n$ vertices, we say that $G$ and $P$ are \textbf{consistent} if and only if $G$ contains a Hamiltonian path that generates $P$. (We may also say that $G$ is consistent with $P$ or that $P$ is consistent with $G$.) The set of all preference matrices consistent with $G$ is called the \textbf{consistency set} of $G$, denoted by $\mathcal{C}(G)$.
\end{defn}


Note that the consistency set of a graph $G$ depends on both the structure of $G$ and its labeling. Indeed, it is possible for two different labelings of the same graph to yield entirely different consistency sets.

It is also important to note that finding the entire consistency set of a graph is, in general, a computationally demanding task---particularly for larger graphs. Indeed, simply determining whether or not a Hamiltonian path \emph{exists} on a given graph is an NP-complete problem \cite{wolfram-ham}.
However, there are some interesting classes of graphs for which it is possible to enumerate all Hamiltonian paths; in particular, it is helpful when the structure of a graph suggests a natural labeling corresponding to the outcomes of an appropriately-sized election.


One such example is the $n$-dimensional hypercube graph (with $2^n$ vertices), labeled so that each pair of adjacent vertices differ on exactly one question (or \emph{bit}). Since this labeling is often called the \emph{Gray code} labeling \cite{wolfram-gray}, we will use the notation $G_n$ to denote such graphs, which we will call \emph{Gray graphs}. Elements of the consistency set $\mathcal{C}(G_n)$ will then be referred to as \emph{cubic} preferences.

To illustrate, note that the graph shown in Figure \ref{fig-cube-path} is in fact $G_3$, and the matrix $C$ is one of many elements of $\mathcal{C}(G_3)$. For an arbitrary $P\in\mathcal{C}(G_n)$, each row of $P$ differs by exactly one bit from the rows immediately preceding and following it. Therefore, given a subset $S$ of $Q_n$, we say that $S$ \emph{changes} from one row to the next if the column containing the bit that differs between the two rows is in $S$. Moreover, if the outcomes on $S$ for these two differing rows are $x$ and $y$ respectively, with $x$ above $y$, we say that $S$ \emph{changes from} $x$ \emph{to} $y$.

With these definitions, we may now begin to study cubic preferences in earnest. We begin by partitioning $\mathcal{C}(G_n)$ into natural equivalence classes, and then enumerate all elements of $\mathcal{C}(G_n)$ for small $n$.


\section{Properties of Cubic Preferences}\label{sec-props}

\subsection{Character and Path Classes}\label{sec-compute}

In order to construct $\mathcal{C}(G_n)$, it will be convenient for us to consider characters that differ only by a permutation of the questions to be equivalent, or \emph{isomorphic}. We define character isomorphism formally below.


\begin{defn}\label{char-iso}
Let $\sigma \in S_n$ (the symmetric group of degree $n$), and let $C$ be any collection of subsets of $Q_n$. Then we define \[ \sigma(C)=\{\sigma(S):S\in C\}, \] where $\sigma(S)=\{\sigma(i):i\in S\}$. Moreover, two characters $C_1$ and $C_2$ are said to be \textbf{isomorphic} if and only if there exists some $\sigma \in S_n$ such that $C_2=\sigma(C_1)$.
\end{defn}

Since character isomorphism is clearly an equivalence relation, it partitions any set of characters into equivalence classes. Likewise, character isomorphism also induces an equivalence relation on any set of preference matrices. The corresponding equivalence classes are called \emph{character classes}, defined formally below.

\begin{defn}\label{def-char-class}
Let $\mathcal{M}$ be a collection of preference matrices, and let $\sim$ be the equivalence relation defined on $\mathcal{M}$ by $P_1 \sim P_2$ if and only if $\character(P_1)$ and $\character(P_2)$ are isomorphic. A \textbf{character class} is an equivalence class of the $\sim$ relation.
\end{defn}

Partitioning preference matrices into character classes simplifies the task of enumerating the characters corresponding to a given graph's consistency set. In the case of $\mathcal{C}(G_n)$, this set can be further partitioned---into what we will call \emph{path classes}---by appealing to the algebraic properties of the automorphism group of $G_n$. 

The automorphism group of the $n$-dimensional hypercube is the hyperoctahedral group, which we represent as the subgroup $H_n$ of $GL_n(\mathbb{Z})$ consisting of all $n \times n$ matrices whose only nonzero entries are $\pm 1$ and for which each row and each column contains exactly one nonzero entry. With a small adjustment to the entries of the matrices in $\mathcal{C}(G_n)$---specifically, replacing each $0$ with $-1$---we can show that $H_n$ acts on $\mathcal{C}(G_n)$ by right multiplication, which intuitively corresponds to permuting and taking bitwise complements of the columns. As an example, let \[ M = \begin{pmatrix} 0 & 0 & -1 \\ 1 & 0 & 0 \\ 0 & 1 & 0 \end{pmatrix} \text{~~~~~and~~~~~} P = \begin{pmatrix} 1 & 1 & 1 \\ 1 & 1 & 0 \\ 1 & 0 & 0 \\ 1 & 0 & 1 \\ 0 & 0 & 1 \\ 0 & 1 & 1 \\ 0 & 1 & 0 \\ 0 & 0 & 0  \end{pmatrix}. \] Replacing each $0$ with $-1$ in $P$ and then multiplying on the right by $M$ yields \[ \begin{pmatrix} 1 & 1 & 1 \\ 1 & 1 & -1 \\ 1 & -1 & -1 \\ 1 & -1 & 1 \\ -1 & -1 & 1 \\ -1 & 1 & 1 \\ -1 & 1 & -1 \\ -1 & -1 & -1  \end{pmatrix} \begin{pmatrix} 0 & 0 & -1 \\ 1 & 0 & 0 \\ 0 & 1 & 0 \end{pmatrix} = \begin{pmatrix} 1 & 1 & -1 \\ 1 & -1 & -1 \\ -1 & -1 & -1 \\ -1 & 1 & -1 \\ -1 & 1 & 1 \\ 1 & 1 & 1 \\ 1 & -1 & 1 \\ -1 & -1 & 1  \end{pmatrix}. \] Now replacing each $-1$ with $0$ in the product yields the matrix \[ P' = \begin{pmatrix} 1 & 1 & 0 \\ 1 & 0 & 0 \\ 0 & 0 & 0 \\ 0 & 1 & 0 \\ 0 & 1 & 1 \\ 1 & 1 & 1 \\ 1 & 0 & 1 \\ 0 & 0 & 1  \end{pmatrix}. \]

Therefore, we see that the action of $M$ on $P$ was to apply the permutation $(132)$ to the columns of $P$ and then take the bitwise complement of the new third column. (Note that the order in which these two components of the action are performed is important.)

Formally, for any matrix $P$ with entries consisting only of $0$ and $\pm 1$, let $\overline{P}$ denote the matrix obtained by replacing each $0$ with $-1$ and each $-1$ with $0$. Note that $\overline{\overline{P}} = P$. Then define the map $\phi: \mathcal{C}(G_n) \times H_n \rightarrow \mathcal{C}(G_n)$ by $\phi(P, M) = \overline{\overline{P}M}$. 

\begin{lem}\label{action}
The map $\phi$ as defined above is a (right) group action of $H_n$ on $\mathcal{C}(G_n)$.
\end{lem}

\begin{proof} By the definition of a group action, we must establish that $\phi$ satisfies both the identity and compatibility (associativity) axioms. For the former, let $I_n$ denote the identity matrix.  Then for all $P \in \mathcal{C}(G_n)$, \[ \phi(P, I_n) = \overline{\overline{P} I_n} = \overline{\overline{P}} = P.\]  For compatibility, note that if $M_1$ and $M_2$ are elements of $H_n$, then associativity of matrix multiplication implies that \begin{align*} \phi(P, M_1 M_2) &= \overline{\overline{P} (M_1 M_2)} \\ &= \overline{(\overline{P} M_1) M_2} \\ &= \overline{\overline{\phi(P, M_1)}M_2} \\ &= \phi(\phi(P, M_1), M_2) \end{align*} for all $P \in \mathcal{C}(G_n)$. This completes the proof.\end{proof}

We will use standard $\cdot$ notation to denote the action defined by $\phi$, abrreviating $\phi(P, M)$ by $P \cdot M$. Recall that, for a group $G$ acting on a set $X$, the \emph{orbit} of an element $x\in X$ is the set $\{x\cdot g : g\in G\}.$ In the case of the hyperoctahderal group $H_n$ acting on $\mathcal{C}(G_n)$, the orbits include matrices generated by paths that differ from one another only by symmetries of the hypercube.

\begin{defn}\label{def-pathclass}
For the group action $\phi$, the orbits of $\mathcal{C}(G_n)$ under $\phi$ are called \textbf{path classes} of $G_n$.
\end{defn}

The next theorem elucidates the relationship between character classes and path classes.

\begin{thm}\label{path-class-proof}
Let $A, B \in \mathcal{C}(G_n)$ be in the same path class. Then $A$ and $B$ are also in the same character class. 
\end{thm}

\begin{proof}
First note that any element $M$ of $H_n$ can be written as a product $M_{\sigma} M_{\tau}$, where $M_\sigma$ is a permutation matrix, and $M_{\tau}$ is a signature matrix---that is, a matrix of the form \[ M_{\tau} = \begin{pmatrix} \pm 1 & 0 & \cdots & 0 & 0 \\ 0 & \pm 1 & \cdots & 0 & 0 \\ \vdots & \vdots & \ddots & \vdots & \vdots \\ 0 & 0 & \cdots & \pm 1 & 0 \\  0 & 0 & \cdots & 0 & \pm 1 \end{pmatrix}. \]

Since $A$ and $B$ are in the same path class, there exists $M = M_{\sigma} M_{\tau} \in H_n$ such that $A \cdot M = B$.  Let $C = \character(A)$, and let $C' = \character(A \cdot M_\sigma)$. Since $M_\sigma$ is a permutation matrix, $C$ and $C'$ are isomorphic. Moreover, $M_\tau$ acts on $ A \cdot M_\sigma$ by taking bitwise complements of some subset of the columns of $A \cdot M_\sigma$---an action that has no impact on separability and therefore preserves character. It follows that \begin{align*} \character(B) &= \character(A \cdot M) \\ &= \character(A \cdot (M_\sigma M_\tau)) \\&= \character((A \cdot M_\sigma) \cdot M_\tau) \\ &= \character(A \cdot M_\sigma) \\&= C', \end{align*} which proves that $A$ and $B$ are in the same character class.
\end{proof}

Note that the converse of Theorem \ref{path-class-proof} is not true; that is, a character class may contain several path classes (as shown in Tables \ref{tableG4p} and \ref{tableG5p}). In general, the collection of path classes is a refinement of the partition of $\mathcal{C}(G_n)$ formed by character classes.  Moreover, each of the path classes for a given $G_n$ has the same cardinality, as demonstrated by the following results.


\begin{lem}\label{gray-col}
For every preference matrix $P$, the columns of $P$ are distinct. Furthermore, no pair of columns of $P$ are bitwise complements of each other.
\end{lem}

\begin{proof}
Suppose two columns of $P$ are identical. Then none of the outcomes that differ on the corresponding questions occur in $P$, and thus $P$ is not a valid preference matrix. A similar argument shows that no pair of columns of $P$ are bitwise complements.
\end{proof}

\begin{thm}\label{orbit-partition}
The path classes of $\mathcal{C}(G_n)$ partition $\mathcal{C}(G_n)$ into equal-sized subsets. Furthermore, each path class has cardinality $|H_n| = 2^n \cdot n!$, and so $\mathcal{C}(G_n)$ has exactly $\frac{|\mathcal{C}(G_n)|}{2^n\cdot n!}$ path classes.
\end{thm}

\begin{proof}
For $P \in \mathcal{C}(G_n)$, we denote its path class as the orbit $\mathcal{O}_P$. Furthermore, we denote the stabilizer subgroup of $P$ in $H_n$ as \[ \mathcal{G}_P = \{  M \in H_n : P \cdot M = P \}. \] We wish to show that, for all $P \in \mathcal{C}(G_n)$, $|\mathcal{O}_P|=2^nn!$. By the Orbit-Stabilizer Theorem and Lagrange's Theorem, \[ |\mathcal{O}_P| = |H_n : \mathcal{G}_P| =  |H_n| / |\mathcal{G}_P|. \] Since we know that $|H_n|= 2^n\cdot n!$, it suffices to show that for all $P \in \mathcal{C}(G_n)$, $\mathcal{G}_P = \{I_n\}$, where $I_n$ denotes the identity matrix.

Let $P \in \mathcal{C}(G_n)$ and suppose $P \cdot M = P$ for some $M \in H_n$. Then $\overline{\overline{P}M} = P$, which implies that $\overline{P}M = \overline{P}$. Now let $M = M_\sigma M_\tau$, where, as before, $M_\sigma$ is a permutation matrix and $M_\tau$ is a signature matrix. Then $\overline{P}M_\sigma M_\tau = \overline{P}$. 

Now suppose $M_\sigma \neq I_n$. Then $\overline{P} M_\sigma \neq \overline{P}$; otherwise, $P$ would have two identical columns, in violation of Lemma \ref{gray-col}. But since we also know that $(\overline{P}M_\sigma)M_\tau = \overline{P}$, this means that $\overline{P}M_\sigma$ and $\overline{P}$ differ only by bitwise complements of columns. This, however, implies that $P$ has two columns that are bitwise complements of each other, again a contradiction to Lemma \ref{gray-col}. It follows that $M_\sigma = I_n$. Thus, $|\mathcal{G}_P| = 1$, from which the result follows.
\end{proof}




These results are important for a number of reasons. First, they show how the separability structures within cubic preferences are related to the hypercube graphs that generate these preferences. Second, they help us to conceive of the consistency set of $G_n$ in a more organized fashion by partitioning it into equal-sized sets with known separability properties. 



\subsection{Enumeration by Direct Computation}

For small values of $n$, direct computation can be used to enumerate the elements of $\mathcal{C}(G_n)$. Tables \ref{tableG2p}--\ref{tableG5p} list the total number of matrices in $\mathcal{C}(G_n)$, for $2 \leq n \leq 5$, divided into character and path classes. The computations were performed by computer, using a Python script to generate preference matrices in $\mathcal{C}(G_n)$ (at least one per path class) and a program developed by TerHaar (one of the first researchers to study the character admissibility problem; see \cite{hodgeterhaar08}) to calculate the character of each matrix. 

We use a notation for character classes wherein $a,b,c,\ldots$ represent arbitrary, distinct elements of $Q_n$. Furthermore, we omit the trivial subsets for conciseness (except in the case of the class of completely nonseparable characters, which we denote $\emptyset$). For example, in the case of $\mathcal{C}(G_3)$, the characters $\{ \emptyset, \{1\},\{1,2\}, \{1,2,3\} \}$ and $\{ \emptyset, \{2\},\{2,3\}, \{1,2,3\} \}$ are in the character class $\{a\},\{a,b\}$, whereas $\{ \emptyset, \{1\},\{2,3\}, \{1,2,3\} \}$ is in $\{a\},\{b,c\}$. For reference, each table also includes the size of the associated hyperoctahedral group, which by Theorem \ref{orbit-partition} is equal to the number of matrices in each path class.

\medskip

\begin{table}[htbp]
\centering
\caption{Characters and Path Classes of $G_2$\\$\lvert H_2 \rvert=2^2\cdot2!=8$}
\label{tableG2p}
\begin{tabular}{cccc}
\hline
Character Class & Total Matrices & Path Classes\\
\hline
\hline
$\{a\}$ & 8 & 1\\
\hline
\end{tabular}
\end{table}

\begin{table}[htbp]
\centering
\caption{Characters and Path Classes of $G_3$\\$\lvert H_3\rvert=2^3\cdot3!=48$}
\label{tableG3p}
\begin{tabular}{cccc}
\hline
Character Class & Total Matrices & Path Classes\\
\hline
\hline
$\emptyset$		& 48		& 1\\
$\{a\}$	 		& 48		& 1\\
$\{a\}\{a,b\}$		& 48	 	& 1\\
\hline
\end{tabular}
\end{table}

\begin{table}[htbp]
\centering
\caption{Characters and Path Classes of $G_4$\\$\lvert H_4 \rvert=2^4\cdot4!=384$}
\label{tableG4p}
\begin{tabular}{cccc}
\hline
Character Class & Total Matrices & Path Classes\\
\hline
\hline
$\emptyset$		& 70272		& 183\\
$\{a\}$	 		& 17280	 	& 45\\
$\{a\}\{a,b\}$		&2688		 	& 7\\
$\{a\}\{a,b,c\}$		&384		   & 1\\
$\{a\}\{a,b\}\{a,b,c\}$	&384			& 1\\
$\{a,b,c\}$  			&384			& 1\\
\hline
\end{tabular}
\end{table}

\begin{table}[htbp]
\centering
\caption{Characters and Path Classes of $G_5$\\$\lvert H_5 \rvert=2^5\cdot5!=3840$}
\label{tableG5p}
\begin{tabular}{ccc}
\hline
Character Class & Total Matrices  & Path Classes\\
\hline
\hline
$\emptyset$ 				& 182278152960	& 47468269\\
$\{a\}$ 					& 5152462080	& 1341787\\
$\{a\}\{a,b\}$ 				& 66666240		& 17361\\
$\{a,b,c,d\}$ 				& 702720		& 183\\
$\{a\}\{a,b\}\{a,b,c\}$ 			& 487680		& 127\\
$\{a\}\{a,b,c\}$ 				& 487680		& 127\\
$\{a,b,c\}$ 					& 487680		& 127\\
$\{a\}\{a,b,c,d\}$ 				& 172800		& 45\\
$\{a\}\{a,b\}\{a,b,c,d\}$ 			& 26880		& 7\\
$\{a\}\{a,b\}\{a,b,c\}\{a,b,c,d\}$ 	& 3840 		& 1\\
$\{a\}\{a,b,c\}\{a,b,c,d\}$ 		& 3840		& 1\\
$\{a,b,c\}\{a,b,c,d\}$ 			& 3840		& 1\\
\hline
\end{tabular}
\end{table}

\subsection{Separability Properties}\label{sec-nested}



Each of the characters in Tables \ref{tableG2p}--\ref{tableG5p} is \emph{nested}, meaning that for each pair of separable subsets $A$ and $B$, either $A \subseteq B$ or $B \subseteq A$. We will now show that this property holds for cubic preferences in general. (See Theorem \ref{nested}.)  We begin with a few preliminary lemmas. Lemma \ref{intersect} is originally due to Bradley, Hodge, and Kilgour \cite{bradley05}, and Lemma \ref{firstrow} is an immediate consequence of the definition of separability.

\begin{lem}\label{intersect} For any preference matrix $P$, $\character(P)$ is closed under intersections. That is, if $A, B \in \character(P)$, then $A \cap B \in \character(P)$. \end{lem}

\begin{lem}\label{firstrow} Let $P\in\mathcal{C}(G)$, and let $S$ be a proper, nonempty subset of $Q_n$ such that $S \in \character(P)$. Furthermore, let $x$ be the outcome on $S$ corresponding to the first row of $P$. Then the first time any outcome on $Q_n-S$ occurs in $P$, the outcome on $S$ in the corresponding row must be $x$. \end{lem}


\begin{lem}\label{firstbitnonsep} Let $P\in\mathcal{C}(G)$ and let $S$ be a proper, nonempty subset of $Q_n$ such that $S \in \character(P)$. Then the bit that changes from the first to the second row of $P$ is not in $S$.\end{lem}

\begin{proof}
Assume, to the contrary, that $S$ is separable and $S$ contains the first bit to change in $P$. Then $$P=\begin{pmatrix} x_1 & y_1 \\ x_2 & y_1 \\ \vdots & \vdots \\ x_j & y_1 \\ x_j & y_2 \end{pmatrix},$$ where $x_1$ and $x_2$ are outcomes on $S$, $y_1$ and $y_2$ are outcomes on $Q_n - S$, and $x_j$ is the outcome on $S$ that occurs in the first pair of rows in which $Q_n-S$ changes.\footnote{For ease of notation, we permute the columns of $P$ (both here and in subsequent results) so that the columns corresponding to $S$ and $Q_n - S$ are grouped together. This notational simplification has no impact on the substance of our arguments.} Note that the first row of $P^{[Q_n-S,y_2]}$ is $x_j$ and the first row of $P^{[Q_n-S,y_1]}$ is $x_1$. Because $S$ is separable, this means that $x_j=x_1$, which is a contradiction since $x_1 y_1$ would then appear twice in $P$. Therefore, if $S$ is separable, the first bit to change is not in $S$.
\end{proof}


We are now ready to state and prove our main result about the characters of cubic preferences.

\begin{thm}\label{nested} Let $P\in\mathcal{C}(G_n)$. If $A, B \in \character(P)$, then $A\subseteq B$ or $B\subseteq A$.\end{thm}

\begin{proof}
We will use a proof by contradiction, divided into cases. Let $A,B\in \character(P)$ such that $A\not\subseteq B$ and $B\not\subseteq A$. Let $Z=A\cap B$, $C=Q_n-(A\cup B)$, $A'=A- Z$, and $B'=B-Z$. By Lemma \ref{firstbitnonsep}, the first bit to change can be in neither $A$ nor $B$. Hence,  the first bit to change must be in $C$. Therefore, $C\neq\emptyset$. This is a contradiction (and therefore establishes the result) when $n = 2$. Throughout the remainder of the proof, we will assume $n\geq 3$ and use the notation $a_i$, $b_i$, $c_i$, and $z_i$ to denote outcomes on $A'$, $B'$, $C$, and $Z$, respectively. 


 Consider the following two cases:

\textbf{Case 1.} $Z=\emptyset$. Note then that $A'=A$ and $B'=B$. We already know that the first bit to change must be in $C$. Without loss of generality, assume that $A$ changes before $B$. Then $P$ has the following form, where $a_2b_1c_k$ is the highest row in which $a_2$ appears, for some $1<k<2^{|C|}$:
\[ P =
\begin{pmatrix}
a_1 & b_1 & c_1 \\
a_1 & b_1 & c_2 \\
\vdots & \vdots & \vdots \\
a_1 & b_1 & c_k\\
a_2 & b_1 & c_k\\
\vdots & \vdots & \vdots\\
a_j & b_2 & c_h\\
\vdots & \vdots & \vdots\\
\end{pmatrix}.
\]

Let $a_jb_2c_h$ be the highest row in which $b_2$ appears, which must appear below the row $a_2b_1c_k$ by the assumption that $A$ changes before $B$. Then by Lemma \ref{firstrow} and the fact that $A$ is separable, $j = 1$.
In other words, before $B$ changes to $b_2$, $A$ must change back to $a_1$.  Let $a_1b_1c_p$ be the first row after $a_2b_1c_k$ in which $A$ changes back to $a_1$, where $c_p$ is some outcome on $C$ so that the row $a_1b_1c_p$ has not yet appeared in $P$. Since only one bit changes at a time and, by assumption, this change occurs in $A$, the row immediately above $a_1b_1c_p$ is $a_zb_1c_p$, for some $z \neq 1$. Hence P has the following form:
 \[ P =
\begin{pmatrix}
a_1 & b_1 & c_1 \\
a_1 & b_1 & c_2 \\
\vdots & \vdots & \vdots \\
a_1 & b_1 & c_k \\
a_2 & b_1 & c_k\\
\vdots & \vdots & \vdots\\
a_z & b_1 & c_p\\
a_1 & b_1 & c_p\\
\vdots & \vdots & \vdots\\
\end{pmatrix}.
\]

Fixing the outcome $b_1 \, c_p$ on $B\cup C$ induces a preference of $a_z \succ a_1$ on $A$. But by the separability of $A$, $a_1$ must be the most preferred outcome on $A$ for any outcome fixed on $B\cup C$. Thus we have a contradiction, which rules out this case.

\textbf{Case 2.} We now consider the case that $Z\neq \emptyset$. As before, the first bit to change must be in $C$, but now that $Z\neq \emptyset$, we have:

 \[ P =
\begin{pmatrix}
a_1&z_1&b_1&c_1\\
a_1&z_1&b_1&c_2\\
\vdots&\vdots&\vdots&\vdots
\end{pmatrix}
\]

Note that, by Lemma \ref{intersect}, $Z=A\cap B$ is separable. Therefore, by Lemma \ref{firstrow}, the first time an outcome on $Q_n-A$, $Q_n-B$, or $Q_n-Z$ occurs, the corresponding outcome on $A$, $B$, or $Z$ must be $a_1z_1$, $z_1b_1$, or $z_1$, respectively. Without loss of generality, assume that $A'$ changes for the first time before $B'$. Let $a_2$ and $b_2$ be the second outcomes that occur on $A'$ and $B'$, respectively. Note that it is possible for $Z$ to change before either $A'$ or $B'$ changes; however, $Z$ will have to change back to $z_1$ prior to the first time either $A'$ or $B'$ changes (by Lemma \ref{firstrow}, as noted above, and the fact that only one bit can change at a time). Therefore, $P$ has the following form, where $c_3$ and $c_4$ are distinct outcomes on $C$ not equal to $c_1$.

 \[ P =
\begin{pmatrix}
a_1&z_1&b_1&c_1\\
a_1&z_1&b_1&c_2\\
\vdots&\vdots&\vdots&\vdots\\
a_1&z_1&b_1&c_3\\
a_2&z_1&b_1&c_3\\
\vdots&\vdots&\vdots&\vdots\\
a_1&z_1&b_1&c_4\\
a_1&z_1&b_2&c_4\\
\vdots&\vdots&\vdots&\vdots
\end{pmatrix}
\]

Note that it is possible that $c_3=c_2$, in which case the second and third rows explicitly given here would become a single row. However, $c_4=c_2$ is impossible.

Now $A'$ must change back to $a_1$ between row $a_2z_1b_1c_3$ and row $a_1z_1b_1c_4$. This change cannot occur immediately after row $a_2 z_1 b_1 c_3$, since row $a_1  z_1 b_1 c_3$ has already appeared and only one bit can change at a time. Thus, between $a_2 z_1 b_1 c_3$ and $a_1 z_1 b_1 c_4$, there must exist a pair of consecutive rows of the following form, where $a_3 \neq a_1$ (it may or may not be the case that $a_3=a_2$), and $z_2$ and $c_5$ are outcomes on $Z$ and $C$, respectively, which---for the time being---may or may not be distinct from $z_1$ and $c_4$:



\begin{center}
\begin{tabular}{>{$}c<{$}>{$}c<{$}>{$}c<{$}>{$}c<{$}}
a_3&z_2&b_1&c_5\\
a_1&z_2&b_1&c_5\\
\end{tabular}
\end{center}

Choose the lowest such pair of rows---that is, the last pair of rows before $a_1z_1b_1c_4$ in which $A'$ changes back to $a_1$. Then $P$ has the following form:

 \[ P =
\begin{pmatrix}
a_1&z_1&b_1&c_1\\
a_1&z_1&b_1&c_2\\
\vdots&\vdots&\vdots&\vdots\\
a_1&z_1&b_1&c_3\\
a_2&z_1&b_1&c_3\\
\vdots&\vdots&\vdots&\vdots\\
a_3&z_2&b_1&c_5\\
a_1&z_2&b_1&c_5\\
\vdots&\vdots&\vdots&\vdots\\
a_1&z_1&b_1&c_4\\
a_1&z_1&b_2&c_4\\
\vdots&\vdots&\vdots&\vdots
\end{pmatrix}
\]

Now consider $z_2$ and $c_5$. If $z_2=z_1$, then fixing an outcome of $b_1c_5$ on $Q_n-A$ induces a preference of $a_3z_1\succ a_1z_1$ on $A$---a contradiction to the separability of $A$ since $a_1 z_1$ is the outcome on $A$ in the top row of $P$. Thus $z_2\neq z_1$. Similarly, $c_5\neq c_4$, as that would induce the order $z_2\succ z_1$ when fixing $a_1b_1c_4$, contradicting the separability of $Z$. Furthermore, $c_5$ must be an outcome which has appeared at least once in a row above $a_3z_2b_1c_5$, since otherwise fixing $b_1c_5$ would yield a most preferred outcome of $a_3z_2$ on $A$, again contradicting the separability of $A$.

Now, by construction, the row immediately above $a_1z_1b_1c_4$ must be $a_1z_ib_1c_j$, for some $i$ and $j$. We first show that $z_i=z_1$.
For contradiction,  suppose $z_i\neq z_1$. Then it must be that $c_j=c_4$, so this row is $a_1z_ib_1c_4$. But then fixing $a_1b_1c_4$ on $Q_n-Z$ induces the order $z_i\succ z_1$ on $Z$, where $z_i\neq z_1$. Since $z_1$ is the outcome on $Z$ in the top row, this violates the separability of $Z$. Hence $z_i=z_1$.

Because rows cannot repeat, we see that $c_j\neq c_1,c_2,c_3,c_4$. Furthermore, if $c_j=c_5$, then fixing $a_1b_1c_5$ on $Q_n-Z$ would induce $z_2\succ z_1$ on $Z$, again contradicting the separability of $Z$. Therefore we set $c_j=c_6\neq c_1,c_2,c_3,c_4,c_5$. This gives us the following outcomes on the lower rows of the matrix $P$:

\[
\begin{pmatrix}
\vdots&\vdots&\vdots&\vdots\\
a_3&z_2&b_1&c_5\\
a_1&z_2&b_1&c_5\\
\vdots&\vdots&\vdots&\vdots\\
a_1&z_1&b_1&c_6\\
a_1&z_1&b_1&c_4\\
a_1&z_1&b_2&c_4\\
\vdots&\vdots&\vdots&\vdots
\end{pmatrix}
\]

Now consider the row immediately above $a_1z_1b_1c_6$. Applying the argument of the preceding two paragraphs with $c_6$ in place of $c_4$ shows that this row must be equal to $a_1z_1b_1c_7$ for some new outcome $c_7$ on $C$. And in fact, this pattern continues for each successive row above $a_1z_1b_1c_4$. Therefore, we must eventually obtain the following pair of rows in $P$, where $c_j \neq c_5$:

\[
\begin{pmatrix}
\vdots&\vdots&\vdots&\vdots\\
a_1&z_2&b_1&c_5\\
a_1&z_1&b_1&c_j\\
\vdots&\vdots&\vdots&\vdots
\end{pmatrix}
\]
But since $z_1\neq z_2$ (as shown previously), we have two bits changing at the same time, a contradiction to the fact that $P \in \mathcal{C}(G_n)$.

We have now shown that in both cases ($Z = \emptyset$ and $Z \neq \emptyset$), the assumption that $A\not\subseteq B$ and $B\not\subseteq A$ leads to a contradiction; therefore, it must be the case that $A\subseteq B$ or $B\subseteq A$.
\end{proof}

The following corollaries follow immediately from Theorem \ref{nested}.

\begin{coro}\label{setcompsep}
Let $P \in \mathcal{C}(G_n)$, and let $S$ be a nonempty, proper subset of $Q_n$. Then $S$ and $Q_n-S$ cannot both be separable.
\end{coro}



\begin{coro} \label{nocompletesep} For all $n \geq 2$ and all $P \in \mathcal{C}(G_n)$, $P$ is not completely separable. \end{coro}



Theorem \ref{nested} places narrow limits on the kinds of characters that can be associated with cubic preferences. In particular, the nested structure of cubic characters sets them apart from the characters generated by Hodge, Krines, and Lahr's \cite{hodgekrineslahr09} method of preseparable extensions, which always contain both a nontrivial subset of $Q_n$ \emph{and} its complement. Therefore, any method to construct cubic preference matrices will necessarily yield different separability structures than those obtained by previous work. In the next section, we introduce two such methods.

It is important to note that the necessary condition provided by Theorem \ref{nested} is not sufficient. For example, in the $n = 5$ case, none of the (nested) character classes $\{a, b\}$, $\{a, b\}\{a,b,c\}$, $\{a,b\}\{a,b,c,d\}$, or $\{a,b\}\{a,b,c\}\{a,b,c,d\}$ appear in Table \ref{tableG5p}. Thus, it remains an open question to completely classify cubic preferences for arbitrary $n$ according to their path and/or character classes.

\section{Constructing Cubic Preferences}\label{sec-construct}


 
Given the rapid growth of the size of $\mathcal{C}(G_n)$ as $n$ increases, it is infeasible in general to construct all cubic preference matrices by brute force computation.
Here, we address this issue by introducing two functions---called \emph{stack} and \emph{weave}---which generate cubic preferences recursively by mapping elements of $\mathcal{C}(G_n)$ to $\mathcal{C}(G_{n+1})$.  These functions behave predictably in that the characters of the constructed matrices are completely determined by the characters of the input matrices. While the stack and weave functions do not generate all possible cubic preferences---and thus, they do not completely characterize the consistency sets of Gray graphs---they do allow us to construct matrices with predictable characters, and thus to study cubic preferences in a more concrete way.

\subsection{The Stack Function}

The first function, which we call the \emph{stack} function, combines pairs of matrices in $\mathcal{C}(G_{n})$ to form a single matrix in $\mathcal{C}(G_{n+1})$. To illustrate the intuitive idea behind the stack function, we will consider a simple example.

First, we choose two matrices, $P_1$, $P_2 \in \mathcal{C}(G_n)$, such that the first row of $P_2$ is equal to the last row of $P_1$. Next, we stack $P_1$ on top of $P_2$ to form a new $2^{n+1} \times n$ matrix. Finally, we insert a column vector of the form $$\begin{pmatrix} 1 &1 & \cdots &0 &0\end{pmatrix}^\top \hspace{3mm}\text{ or }\hspace{3mm}\begin{pmatrix} 0&0& \cdots &1 &1\end{pmatrix}^\top,$$ either as the first or last column of the matrix or in between two other columns. 

Applying this process with $$P_1=\begin{pmatrix} 0 & 0 \\ 0 & 1 \\ 1 & 1 \\ 1 & 0 \end{pmatrix}, ~P_2=\begin{pmatrix} 1 & 0 \\ 0 & 0 \\ 0 & 1 \\1 &1 \end{pmatrix},$$ and the column vector $\mathbf{c} = \begin{pmatrix} 1 & 1 & 1 & 1 & 0 & 0 & 0 & 0\end{pmatrix}^\top$ inserted as the second column, we obtain the matrix
$$\begin{pmatrix} 0 & 1 & 0 \\ 0 & 1 & 1 \\ 1 &1 & 1 \\ 1 &1 & 0 \\ 1 & 0 & 0 \\ 0 & 0 & 0 \\ 0 &0  & 1 \\ 1 & 0 &1\end{pmatrix},$$ which is an element of $\mathcal{C}(G_3)$.





The formal definition of the stack function makes this intuitive process more precise. In order to ensure a well-defined domain, we must first define which matrices can be stacked on top of one another.

\begin{defn}
Let $A$ and $B \in \mathcal{C}(G_n)$. We say that $A$ is \textbf{stackable} over $B$ (or $B$ is stackable under $A$) if and only if the last row of $A$ is identical to the first row of $B$. We define the \textbf{stackability set} of a matrix $A$ to be the set $$\mathcal{S}(A) = \{B \in \mathcal{C}(G_n) : A \text{ is stackable over } B \}.$$
\end{defn}

We now provide the formal definition of the stack function.

\begin{defn}\label{def-stack}
Let $A\in \mathcal{C}(G_n )$, $B\in \mathcal{S}(A)$, and $k \in \{1,\ldots,n+1\}.$ We define $A\stack_k^1 B = C$ to be the $(2^{n+1})\times( n+1 )$ matrix whose entries are given by:

$$
c_{i,j} = \left\{
        \begin{array}{lll}
	a_{i,j},		& 1\leq i\leq2^n 			&\text{and } j<k \\
	a_{i,j-1},		& 1\leq i\leq2^n 			&\text{and } j>k \\
	b_{i-2^n,j},		& 2^n +1\leq i\leq 2^{n+1}	&\text{and } j<k \\
	b_{i-2^n,j-1},	& 2^n +1\leq i\leq 2^{n+1}	&\text{and } j>k \\
	1,			& 1\leq i \leq 2^n 			&\text{and } j=k \\
	0,			& 2^n + 1 \leq i\leq 2^{n+1}		&\text{and } j=k 			
        \end{array}
    \right.
$$
where $a_{i,j}$ and $b_{i,j}$ denote the entries of $A$ and $B$, respectively. We define $A\stack_k^0 B$ in the same manner as $A\stack_k^1 B$, except with the bitwise complement taken of the $k^{th}$ column.
\end{defn}

Since separability is unaffected by taking bitwise complements of columns, we need only consider $A\stack_k^1 B$ in the results that follow, as similar arguments apply to $A\stack_k^0 B$.  For convenience, given any set $T \subseteq Q_n$ and any $k$ with $1 \leq k \leq n+1$, we will define the notation $T^k$ as follows: \[ T^k = \{q\in T : q<k\}\cup\{q+1 : q\in T \text{ and } q \geq k\}. \]
For example, $\{1, 2, 3, 5\}^3 = \{1, 2, 4, 6\}$.

We will now show that for all $A\in\mathcal{C}(G_n)$ and all $B\in\mathcal{S}(A)$, the character of $C=A\stack_k^1 B$ is uniquely determined by $\character(A)$, $\character(B)$, and $k$. Moreover, we will demonstrate explicitly how $\character (C)$ may be computed given this information. This proof involves several intermediate results, beginning with the lemma below.



\begin{lem} \label{gen-nonsep}
Let $P_1 \in \mathcal{C}(G_n)$ and $P_2\in\mathcal{S}(P_1)$. If $T \subseteq Q_{n+1}$ with $T \neq \emptyset$ and $k \notin T$, then $T \notin \character(P_1\stack_k^1 P_2)$.

\end{lem}

\begin{proof} Without loss of generality, assume $k = n+1$. Let $T \subseteq Q_{n+1}$ be given, with $T \neq \emptyset$ and $k \notin T$. Then $T \subseteq Q_n$. By the definition of the stack function, $P_1\stack_k^1 P_2$ has the form
 \[
P_1\stack_k^1 P_2 = \begin{pmatrix}
\vdots&\vdots&\vdots\\
x_1 & y & 1 \\
\vdots&\vdots&\vdots\\
x_2 & y & 1 \\
x_2 & y &  0 \\
\vdots&\vdots&\vdots\\
x_1 & y &  0 \\
\vdots&\vdots&\vdots\\
\end{pmatrix},
\]
where $x_1$ and $x_2$ are distinct outcomes on $T$ and $y$ is an outcome on $Q_n - T$. (Note that if $T = Q_n$, and therefore $Q_n - T$ is empty, the remainder of the proof still follows, simply ignoring the middle column of $P_1\stack_k^1 P_2$.) Fixing an outcome of $(y,1)$ on $Q_{n+1} - T = (Q_n - T) \cup \{n+1\}$ induces a preference of $x_1 \succ x_2$ on $T$, while fixing an outcome of $(y,0)$ on $Q_{n+1} - T$ induces a preference of $x_2 \succ x_1$ on $T$. Therefore, $T$ is not separable with respect to $P_1\stack_k^1 P_2$, and $T \notin \character(P_1\stack_k^1 P_2)$.
%
%
%
%
\end{proof}


Lemma \ref{gen-nonsep} implies that the stack function preserves nonseparability. Indeed, for stackable matrices $P_1$ and $P_2$, every element of $\character(P_1\stack_k^1 P_2)$ must include the question, $k$, corresponding to the added column. Therefore, if $T$ is nonseparable with respect to $P_1$ or $P_2$, then $T^k$ is nonseparable with respect to $P_1\stack_k^1 P_2$. (The same conclusion could be made if $T$ was separable, but the point here is that the property of nonseparability is in fact preserved.) Since we have established which sets must be nonseparable in a matrix produced by the stack function, we will now turn our attention to the sets that can be separable.


\begin{lem}\label{aug-sep}
Let $P_1 \in \mathcal{C}(G_n)$ and $P_2\in\mathcal{S}(P_1)$, with characters $C_1$ and $C_2$ respectively. Let $1 \leq k \leq n+1$. Then $T \in C_1 \cap C_2$ if and only if $T^k \cup \{k\} \in \character(P_1\stack_k^1 P_2)$.
\end{lem}

\begin{proof}
We begin by proving that if $T^k \cup \{k\} \in \character(P_1\stack_k^1 P_2)$, then $T \in C_1 \cap C_2$. We will do so by proving the contrapositive: if $T \notin C_1 \cap C_2$, then $T^k \cup \{k\} \notin\character(P_1\stack_k^1 P_2)$.

Suppose $T \notin C_1 \cap C_2$. Without loss of generality, let $T\notin C_1$.  Then there exist outcomes $x$ and $y$ on $Q_n-T$ such that 
$$P_1 ^{[Q_n-T,x]}\neq P_1^{[Q_n-T,y]}.$$ 
Without loss of generality, let $k=n+1$. Then $T^k = T$ and $Q_{n+1} - (T^k \cup \{k\}) = Q_n - T$. Let $$A=(P_1\stack_k^1 P_2)^{[Q_{n+1}-(T^k \cup \{k\}),x]}$$ and $$B=(P_1\stack_k^1 P_2)^{[Q_{n+1}-(T^k \cup \{k\}),y]}.$$ Then, by construction, it may be seen that
$$A=
\begin{pmatrix}
P_1^{[Q_n-T, x]} & \vec{1} \\
P_2^{[Q_n-T, x]} & \vec{0} 
\end{pmatrix}
\hspace{3mm}\text{ and }\hspace{3mm}
B=
\begin{pmatrix}
P_1^{[Q_n-T, y]} & \vec{1} \\
P_2^{[Q_n-T, y]} & \vec{0} 
\end{pmatrix},$$
where $\vec{1}$ and $\vec{0}$ represent column vectors of ones and zeros, respectively, of the proper length.

Now, since $P_1 ^{[Q_n-T,x]}\neq P_1^{[Q_n-T,y]}$, it follows that $A \neq B$. Therefore, by the definition of separability, $T^k\cup \{k\}$ is not separable with respect to $P_1 \stack_k ^1 P_2$, and so $T^k\cup \{k\}\notin \character(P_1\stack_k ^1 P_2)$.


Next, we will prove that if $T \in C_1 \cap C_2$, then $T^k \cup \{k\} \in \character(P_1\stack_k^1 P_2)$.

Suppose $T \in C_1 \cap C_2$. Then we have three cases:

\textbf{Case 1:} $T = Q_n$. Then $T^k \cup \{k\} = Q_{n+1}$. Since $P_1 \stack_k^1 P_2$ is a preference matrix on $Q_{n+1}$,  it follows by definition that $Q_{n+1} \in \character (P_1 \stack_k^1 P_2)$.

\textbf{Case 2:} $T = \emptyset$. Then $T^k \cup \{k\} = \{k\}$. Now by construction of $P_1\stack_k ^1 P_2$, $\{k\}$ must be separable, since regardless of the outcome fixed on $Q_{n+1}-\{k\}$, the preference order $1 \succ 0$ is induced on $\{k\}$. Thus we have that $T^k \cup \{k\} \in \character(P_1\stack_k ^1 P_2)$.

\textbf{Case 3:} $T \neq \emptyset$ and $T \neq Q_n$. Since $T \in \character(P_1) \text{ and } T \in \character(P_2)$, it follows that for all outcomes $x$ and $y$ on $Q_n-T$,
$$P_1 ^{[Q_n-T,x]}= P_1^{[Q_n-T,y]} \hspace{3mm}\text{ and }\hspace{3mm}  P_2 ^{[Q_n-T,x]}= P_2^{[Q_n-T,y]}.$$ 
Without loss of generality, let $k=n+1$. Fix arbitrary outcomes $x$ and $y$ on $Q_n -  T$. Let $$A=(P_1\stack_k^1 P_2)^{[Q_{n+1}-(T^k \cup \{k\}),x]}$$ and $$B=(P_1\stack_k^1 P_2)^{[Q_{n+1}-(T^k \cup \{k\}),y]}.$$ As in Case 1,
$$A=
\begin{pmatrix}
P_1^{[Q_n-T, x]} & \vec{1} \\
P_2^{[Q_n-T, x]} & \vec{0} 
\end{pmatrix}
\hspace{3mm}\text{ and }\hspace{3mm}
B=
\begin{pmatrix}
P_1^{[Q_n-T, y]} & \vec{1} \\
P_2^{[Q_n-T, y]} & \vec{0} 
\end{pmatrix}.$$

Now since $P_1 ^{[Q_n-T,x]}= P_1^{[Q_n-T,y]}$ and $P_2 ^{[Q_n-T,x]}= P_2^{[Q_n-T,y]}$, it follows that $A = B$. Since $x$ and $y$ were chosen arbitrarily, the definition of separability implies that $T^k\cup \{k\}$ is separable with respect to $P_1 \stack_k ^1 P_2$. Thus, $T^k\cup \{k\}\in \character(P_1\stack_k ^1 P_2)$.


Having proven the implication in both directions, we have established that $T \in C_1 \cap C_2$ if and only if $T^k \cup \{k\} \in \character(P_1\stack_k^1 P_2)$.
\end{proof}

Lemma \ref{gen-nonsep} establishes that every nonempty set $S \in \character(P_1\stack_k^1 P_2)$ must contain $k$, and therefore can be written in the form $T^k \cup \{k\}$ for some $T \subseteq Q_n$. Moreover, Lemma \ref{aug-sep} demonstrates that a set of this form belongs to $\character(P_1\stack_k^1 P_2)$ if and only if $T$ is an element of both $\character(P_1)$ and $\character(P_2)$. Together, these two results establish the following theorem, which describes the character of the stacked matrix $P_1\stack_k^1 P_2$ in terms of the characters of the two matrices, $P_1$ and $P_2$, used to construct it.

\begin{thm}\label{stack-char}
Let $P_1\in\mathcal{C}(G_n)$ and $P_2\in\mathcal{S}(P_1)$ and let $\character(P_1)=C_1$ and $\character(P_2)=C_2$.  Let $1\leq k \leq n+1$. Then $$\character(P_1\stack_k^1 P_2) = \{T^k \cup \{k\} : T \in C_1\cap C_2\}\cup\{\emptyset\}.$$
\end{thm}

In addition to determining the characters of stacked matrices, the stack function illuminates more general properties of cubic preferences. In particular, observe in Theorem \ref{stack-char} how the question corresponding to the added column $k$ is always separable. In fact, the converse of this result holds as well. In particular, if an individual question is separable with respect to a cubic preference matrix, then that matrix is necessarily stacked.

\begin{thm}\label{stack-col}
Let $P \in \mathcal{C}(G_n)$, and let $k \in Q_n$. Then the following are equivalent:

\begin{enumerate}  \item The set $\{k\}$ is separable with respect to $P$. 
\item  The $k^{th}$ column of $P$ is of the form  
\[\begin{pmatrix} 1 & 1 & 1 & \cdots & 0 & 0 & 0 \end{pmatrix} ^\top=\begin{pmatrix} \vec{1} & \vec{0} \end{pmatrix}^\top, \] 
or its bitwise complement.
\item The matrix $P$ is equal to $P_1 \stack_k^1 P_2$ or $P_1 \stack_k^0 P_2$ for some $P_1 \in \mathcal{C}(G_{n-1})$ and $P_2 \in \mathcal{S}(P_1)$.
\end{enumerate}
\end{thm}

\begin{proof}
Let $\vec{k}$ denote the $k^{th}$ column of $P$. We will prove that $1 \rightarrow 2$, $2 \rightarrow 3$, and $3 \rightarrow 1$.

For $1 \rightarrow 2$, we will prove the contrapositive---that is, if $\vec{k} \neq \begin{pmatrix} \vec{1} & \vec{0} \end{pmatrix}^\top$ (or its bitwise complement), then $\{k\}$ is nonseparable. Assume then that $\vec{k} \neq \begin{pmatrix} \vec{1} & \vec{0} \end{pmatrix}^\top$.

Without loss of generality, let $k = n$, and let the first entry of $\vec{k}$ be $1$. Since $\vec{k} \neq \begin{pmatrix} \vec{1} & \vec{0} \end{pmatrix}^\top$, we know that the top half of the matrix must include at least one row with an entry of $0$ in column $k$. Let $(x,0)$ be the first such row. There must be some row after $(x,0)$ with an entry of $1$ in column $k$. Let $(y,1)$ be the first such row. Note that $x \neq y$ since since each outcome must appear exactly once in $P$ and only one bit can change between any two rows. Thus, $P$ must have the following form:

$$\begin{pmatrix} 
\vdots & \vdots \\
x & 1 \\
x & 0 \\
\vdots & \vdots \\
y & 0 \\
y & 1 \\
\vdots & \vdots \\
\end{pmatrix},$$
where $x$, and $y$ are outcomes on $Q_n - \{k\}$. 

Observe that fixing an outcome of $x$ on $Q_n-\{k\}$ induces a preference of $1 \succ 0$ on $\{k\}$, but fixing an outcome of $y$ on $Q_n-\{k\}$ induces a preference of $0 \succ 1$. Therefore, $\{k\}$ is nonseparable.


For $2 \rightarrow 3$, assume without loss of generality that $k = n$ and $\vec{k} = \begin{pmatrix} \vec{1} & \vec{0} \end{pmatrix}^\top$. Let $P_1 = P^{[\{k\}, 1]}$ and $P_2 = P^{[\{k\}, 0]}$. Then \[ P = \begin{pmatrix}
P_1 & \vec{1}\\
P_2 & \vec{0} 
\end{pmatrix}. \] Since $P \in \mathcal{C}(G_n)$ and the first $2^{n-1}$ rows of $P$ agree on $k$, it must be that each pair of consecutive rows in $P_1$ differ on exactly one question. Thus, $P_1 \in \mathcal{C}(G_{n-1})$. Likewise, $P_2 \in \mathcal{C}{(G_{n-1})}$. Furthermore, since the middle two rows of $P$ differ on $k$, it must be that the last row of $P_1$ is identical to the first row of $P_2$. Therefore, $P_2 \in \mathcal{S}(P_1)$, and it follows that $P = P_1 \stack_k^1 P_2$. 

That $3 \rightarrow 1$ follows immediately from Theorem \ref{stack-char}. \end{proof}



The following corollary is an immediate consequence of Theorem \ref{nested}, but may also be proven independently using Theorem \ref{stack-col} and Lemma \ref{gray-col}.

\begin{coro}
Let $P \in \mathcal{C}(G_n)$, and let $j$, $k \in Q_n$ with $j \neq k$.  If $\{j\} \in \character(P)$, then $\{k\} \notin\character(P)$.
\end{coro}


\subsection{The Weave Function}

We will now turn our attention to a second function, called the \emph{weave} function, which maps each individual matrix in $\mathcal{C}(G_n)$ to a larger matrix in  $\mathcal{C}(G_{n+1})$. As with the stack function, we will begin with an example. 


First, we select a preference matrix $P_1 \in \mathcal{C}(G_n)$. Then we duplicate each row to create a $2^{n+1} \times n$ matrix in which rows $i$ and $i+1$ are equal for all odd $i$. Finally, we insert any column of zeros and ones that preserves the Gray code by allowing only one bit to change from each row to the next.

Applying this process to the matrix $$P_1 = \begin{pmatrix} 0 & 0 \\ 0 & 1 \\ 1 & 1 \\ 1 & 0 \end{pmatrix},$$ with the column vector $\begin{pmatrix} 0 & 1 & 1 & 0 & 0 & 1 & 1 & 0\end{pmatrix}^\top$ inserted as the third column,  we obtain the matrix
$$\begin{pmatrix} 0 & 0 & 0 \\0 & 0 & 1 \\ 0 & 1 & 1 \\0 & 1  & 0\\ 1 & 1 & 0\\1 & 1 & 1 \\ 1 & 0 & 1\\  1 & 0 & 0 \end{pmatrix},$$ which belongs to $\mathcal{C}(G_3)$.

The next definition formalizes this intuitive ``weaving'' process.

\begin{defn}
Let $A$ be any preference matrix in $\mathcal{C}(G_n)$, and let $k \in \{1, \ldots, n+1\}$. 
We define $B = \weave_k^1(A)$ to be the $(2^{n+1})\times(n+1)$ matrix whose entries are given by:

$$
b_{i,j} = \left\{
        \begin{array}{lll}
	a_{i/2,j},		& i \text{ even }						&\text{and } j<k \\
	a_{i/2,j-1},		&  i \text{ even }						&\text{and } j>k \\
	a_{(i+1)/2,j},		& i \text{ odd }						&\text{and } j<k \\
	a_{(i+1)/2,j-1},	&  i \text{ odd }						&\text{and } j>k \\
	1,			& i\equiv 0, 1 ~(\text{\em mod}~4) 	&\text{and } j=k \\
	0,			& i\equiv2, 3 ~(\text{\em mod}~4)   	&\text{and } j=k 			
        \end{array},
    \right.
$$
where $a_{i,j}$ denotes the entries of $A$. We define $\weave_k^0(A)$ in the same manner as $\weave_k^1(A)$, except with the bitwise complement taken of the $k^{th}$ column.
\end{defn}

As with the stack function, we can describe the character of any matrix output by the weave function in terms of the character of the input matrix. As taking the bitwise complement of a column always preserves separability, we need only consider $\weave_k^1(A)$ in the results that follow. Similar results apply to $\weave_k^0(A)$.

\begin{lem}\label{weave-knonsep}Let $P\in\mathcal{C}(G_n)$. For all $S\in\character(\weave_k^1(P))$ with $S\neq Q_{n+1}$, $k\notin S$.
\end{lem}

\begin{proof} Without loss of generality, assume that $k=n+1$. Let $S \subset Q_{n+1}$ with $k \in S$. We will show that $S$ is nonseparable.
Since $S \neq Q_{n+1}$, $Q_{n+1} - S \neq \emptyset$. 
Moreover, since $Q_{n+1} - S$ must change at some point, $\weave_k^1(P)$ must have the following form: 
\[ \weave_k^1(P) = \begin{pmatrix}  \vdots & \vdots & \vdots \\ x_1 & y & z \\ x_1 & y & 1-z \\ x_2 & y & 1 - z \\ x_2 & y & z \\ \vdots & \vdots & \vdots \end{pmatrix}, \]
where $x_1$ and $x_2$ are outcomes on $Q_{n+1} - S$, $y$ is an outcome on $S - \{k\}$, and $z \in \{0, 1\}$. 
Note that fixing $x_1$ on $Q_{n+1} - S$ induces an ordering of $(y, z) \succ (y, 1 - z)$ on $S$. However, fixing $x_2$ on $Q_{n+1} - S$ induces an ordering of $(y, 1 - z) \succ (y, z)$ on $S$. Therefore, $S$ is nonseparable. \end{proof}

Both the stack and weave functions increase the dimension of the resulting matrix by adding a certain type of column. In the case of the stack function, the question corresponding to this added column must be included in every nontrivial separable set. In the case of the weave function, however, the question corresponding to the added column is \emph{never} included in a separable set (with the exception of $Q_{n+1}$).   Interestingly, while the stack function preserves nonseparability, the weave function preserves separability, as shown below.

\begin{lem}\label{weave-preservesep}
Let  $P\in\mathcal{C}(G_n)$, and let $T\subseteq Q_{n}$. Then  $T\in \character(P)$ if and only if $T^k\in \character(\weave_k^1(P))$. 
\end{lem}


\begin{proof}  Let $P_* = \weave_k^1(P)$. If $T = \emptyset$, then the proof is trivial.  Likewise, if $T = Q_n$, note that $Q_n$ is always an element of $\character(P)$, and $Q_n^k$ is always an element of $\character(P_*)$ since $Q_{n+1} - Q_n^k = \{k\}$ and \[ P_*^{[Q_{n+1} - Q_n^k, 0]} = P = P_*^{[Q_{n+1} - Q_n^k, 1]} \] by the definition of the weave function.

Now assume that $T$ is a nonempty, proper subset of $Q_n$. For each direction, we will prove the contrapositive. That is, we will prove that $T \notin \character(P)$ if and only if $T^k \notin \character(P_*)$.

First note that $Q_{n+1} - T^k = (Q_n - T)^k \cup \{k\}$. Also, any outcome on $Q_n - T$ can be viewed as an outcome on $(Q_n - T)^k$ by simply shifting the questions accordingly.  Let $x$ and $z$ be any outcomes on $Q_n -  T$ and $\{k\}$, respectively. As shown in the first paragraph of this proof, $P = P_*^{[Q_{n+1} - Q_n^k,z]} = P_*^{[\{k\},z]}$. Hence, \begin{align*}
P^{[Q_n - T, x]} &= (P_*^{[\{k\},z]})^{[Q_n - T,x]} \\ &= P_*^{[(Q_n - T)^k \cup \{k\},(x,z)]} \\ &= P_*^{[Q_{n+1} - T^k,(x,z)]}.\end{align*}

Now suppose $T \notin \character(P)$. Then there exist outcomes $x$ and $y$ on $Q_n - T$ such that \[ P^{[Q_n - T, x]} \neq P^{[Q_n - T, y]}, \] which implies that \[ P_*^{[Q_{n+1} - T^k, (x, 1)]} \neq P_*^{[Q_{n+1} - T^k, (y, 1)]}, \] and so $T^k \notin \character(\weave_k^1(P))$.

A similar argument establishes the converse. In particular, if \[ P_*^{[Q_{n+1} - T^k, (x, z)]} \neq P_*^{[Q_{n+1} - T^k, (y, z)]} \] for some outcomes $x$, $y$ on $Q_n - T$ and $z$ on $\{k\}$, then \[ P^{[Q_n - T, x]} \neq P^{[Q_n - T, y]}, \] and so $T \notin \character(P)$. \end{proof}

With the preceding lemmas, we can now prove that the character of the matrix output by the weave function is uniquely determined by the character of the input matrix.

\begin{thm}\label{weave-char}
Let  $P\in\mathcal{C}(G_n)$ and let $\character(P)=C$ .
Then, \[\character(\weave_k^1(P))=\{T^k: T\in C\} \cup\{Q_{n+1}\}.\]
\end{thm}

\begin{proof}
%

First, let $S\in\character(\weave_k^1(P))$. We must show that $S=Q_{n+1}$ or that there exists $T\in C$ such that $T^k=S$. Suppose $S\neq Q_{n+1}$. If $S = \emptyset$, then $S= \emptyset^k$ and $\emptyset \in C$. If $S\neq\emptyset$, then Lemma \ref{weave-knonsep} implies that $k \notin S$. Now let \[ T=\{q:q\in S \text{ and } q<k\}\cup\{q-1:q\in S \text{ and } q>k\}.\] Then, $T^k=S$, and by Lemma \ref{weave-preservesep}, $T\in C$.

For the reverse inclusion, first note that if $S=Q_{n+1}$, then $S\in \character(\weave_k^1(P))$ trivially. On the other hand, if $S=T^k$ for some $T\subseteq Q_n$ with $T\in C$, then Lemma \ref{weave-preservesep} implies that $S\in\character(\weave_k^1(P))$, which completes the proof.
\end{proof}

%
%

Like the stack function, the weave function always yields a matrix with a special column---one whose structure allows us to make conclusions about the separability of certain related sets. It is also the case that the existence of such a column guarantees that the matrix in question is in fact woven. The following result, which is analogous to Theorem \ref{stack-col}, formalizes this claim.

\begin{thm}\label{weave-col}
Let $P \in \mathcal{C}(G_n)$, and let $k \in Q_n$. Then the following are equivalent: \begin{enumerate}
\item The set $Q_n - \{k\}$ is separable with respect to $P$.
\item The $k^{th}$ column of $P$ is of the form \[ \vec{w} = \begin{pmatrix} 1 & 0 & 0 & 1 & 1 & \cdots & 0 & 0 & 1 \end{pmatrix}^\top\] or its bitwise complement. 
\item  The matrix $P$ is equal to $\weave_k^1(P_1)$ or $\weave_k^0(P_1)$ for some $P_1 \in \mathcal{C}(G_{n-1})$.
\end{enumerate}
\end{thm}

\begin{proof} We will prove that $1 \rightarrow 2$, $2 \rightarrow 3$, and $3 \rightarrow 1$.

 To prove $1 \rightarrow 2$, let $\vec{k}$ denote the $k^{th}$ column of $P$.  We will assume, to the contrary, that $Q_n - \{k\}$ is separable and that $\vec{k} \neq \vec{w}$ (or its bitwise complement). We will consider cases based on the possible forms of $\vec{k}$. In each case, the same argument applies to the bitwise complement of $\vec{k}$.

\textbf{Case 1:} Suppose $\vec{k}=\begin{pmatrix} \cdots & 1 & 0 & 1 & \cdots \end{pmatrix}^\top$, where the omitted entries may be anything. Then, because $P \in \mathcal{C}(G_n)$ and therefore consecutive rows must differ on exactly one question, $P$ must have the following form for some outcome $z$ on $Q_n-\{k\}$:
$$P = \begin{pmatrix}
\vdots & \vdots \\ 1 & z \\ 0 & z \\ 1 & z \\ \vdots & \vdots
\end{pmatrix}$$

This, however, is a contradiction, since the row $(1, z)$ is repeated.

\textbf{Case 2:} Suppose $\vec{k}=\begin{pmatrix} \cdots & 1 & 1 & 1 & 0 & \cdots \end{pmatrix}^\top$, where, again, the omitted entries may be anything. Then $P$ must be of the following form for some distinct outcomes $x$, $y$, and $z$ on $Q_n-\{k\}$:
$$P =\begin{pmatrix}
\vdots & \vdots \\ 1 & x \\ 1 & y \\ 1 & z \\ 0 & z \\ \vdots & \vdots
\end{pmatrix}$$
Thus, fixing an outcome of $1$ on $\{k\}$ induces the preference order $x \succ y \succ z$ on $Q_n-\{k\}$. Since we have assumed that $Q_n-\{k\}$ is separable, the same preference order must be induced by fixing 0 on $\{k\}$.  Therefore, $P$ must have the following form:
$$P = \begin{pmatrix}
\vdots & \vdots \\ 0 & x \\ \vdots & \vdots \\ 0 & y \\ \vdots & \vdots \\ 1 & x \\ 1 & y \\ 1 & z \\ 0 & z \\ \vdots & \vdots
\end{pmatrix}$$

Observe that if the outcome $0$ occurs on $\{k\}$ in any row between $(0, y)$ and $(1, x)$, then $Q_n-\{k\}$ will no longer be separable. Furthermore, there must exist at least one row between $(0, y)$ and $(1, x)$, in order to preserve the Gray code. The first such row, therefore, must have an outcome of $1$ on $k$ and an outcome of $y$ on $Q_n-\{k\}$, as shown below:
$$\begin{pmatrix}
\vdots & \vdots \\ 0 & x \\ \vdots & \vdots \\ 0 & y \\ 1 & y \\ \vdots & \vdots \\ 1 & x \\ 1 & y \\ 1 & z \\ 0 & z \\ \vdots & \vdots
\end{pmatrix}$$
This, however, is a contradiction, since the row $(1, y)$ is repeated.

\textbf{Case 3:} If $\vec{k}$ does not satisfy the conditions of the previous cases, then it must be the case that there is no entry that is both immediately preceded and followed by different entries (Case 1), nor are there any instances in which three or more consecutive entries are identical (Case 2). This implies that all of the entries in $\vec{k}$ must be preceded by a 1 and followed by a 0, or vice versa. Consider that for any column satisfying this property, all entries---except possibly the first and the last, since they do not have entries both preceding and following them---must come in alternating pairs of identical entries. The weave column $\vec{w}$ satisfies this property, but we have already assumed that $\vec{k} \neq \vec{w}$ (or its bitwise complement). Thus, we must construct a column vector that is not $\vec{w}$, does not contain entries repeated three or more times, and must have pairs of alternating identical entries. The only column vector that satisfies these properties is $\vec{k}=\begin{pmatrix} 1 & 1 & 0& 0 & \cdots & 1 & 1 & 0 & 0 \end{pmatrix}^\top$ or its bitwise complement. Now there are exactly $2^{n-1}$ distinct outcomes on $Q_n-\{k\}$, which we denote $x_1,x_2,\ldots,x_{2^{n-1}}$. In order to preserve the Gray code labeling, these outcomes must be arranged in the following manner within $P$:
$$\begin{pmatrix}
1 & x_1 \\
1 & x_2 \\
0 & x_2 \\
0 & x_3 \\
\vdots & \vdots \\
1 & x_{2^{n-1}-1} \\
1 & x_{2^{n-1}} \\
0 & x_{2^{n-1}} \\
0 & x_1
\end{pmatrix}$$
Observe that fixing an outcome of 1 on $k$ makes $x_1$ the most preferred outcome on $Q_n-\{k\}$, whereas fixing 0 on $k$ makes $x_1$ the least preferred outcome on $Q_n-\{k\}$. Therefore, by definition, $Q_n-\{k\}$ is nonseparable, which contradicts our assumption.

In each of these cases, we obtain a contradiction. Therefore, if $Q_n - \{k\}$ is separable, it must be that $\vec{k}$ is equal to $\vec{w}$ or its bitwise complement.

To establish that $2 \rightarrow 3$, let $P$ be any matrix in $\mathcal{C}(G_n)$ whose $k^{th}$ column has the form $\vec{w} = \begin{pmatrix} 1 & 0 & 0 & 1 & 1 & \cdots & 0 & 0 & 1 \end{pmatrix}^\top$ (or its bitwise complement). In order to preserve the Gray code, rows $i$ and $i + 1$ must agree on $Q_n - \{k\}$ for all odd $i$ (since these rows differ on $\{k\}$).  Thus, we have the equality $P^{[\{k\}, 1]} = P^{[\{k\}, 0]}$.  Call this matrix $P_1$. Then $P = \weave_k^1(P_1)$, as desired.

Finally, for $3 \rightarrow 1$, suppose $P$ is equal to $\weave_k^1(P_1)$ or $\weave_k^0(P_1)$ for some $P_1 \in \mathcal{C}(G_{n-1})$. Since $P_1 \in \mathcal{C}(G_{n-1})$, $Q_{n-1}$ is trivially separable with respect to $P_1$. Moreover, $Q_n - \{k\} = Q_{n-1}^k$, which implies by Theorem \ref{weave-char} that $Q_n - \{k\}$ is separable with respect to $P$.\end{proof}

The following corollary is immediate from Theorem \ref{nested}, but may also be proven independently using Theorem \ref{weave-col} and Lemma \ref{gray-col}.

\begin{coro}
Let $P \in \mathcal{C}(G_n)$, and let $j$, $k \in Q_n$ with $j \neq k$. If $Q_n - \{j\} \in\character(P)$, then $Q_n-\{k\} \notin\character(P)$.
\end{coro}

Given the value of $|\mathcal{C}(G_{n-1})|$, it is relatively straightforward calculate the number of matrices in $|\mathcal{C}(G_n)|$ that are generated by the stack and weave functions. However, the fact that $|\mathcal{C}(G_n)|$ is not known in general makes asymptotic analysis more difficult. Results for small $n$ suggest that the proportion of all matrices in $\mathcal{C}(G_n)$ that can be produced by the stack and/or weave functions approaches $0$ as $n$ increases. (For example, this proportion is approximately $0.667$ for $n = 3$, $0.231$ for $n = 4$, and $0.028$ for $n = 5$.) It is worth noting in particular that, by Theorems \ref{stack-char} and \ref{weave-char}, the stack and weave functions can never produce a completely nonseparable preference matrix. In general, it is known (see \cite{hodgeterhaar08}) that as $n\rightarrow \infty$, the probability of a random preference matrix being completely nonseparable approaches 1. Our computational data (Tables 1--4) suggest that this trend is also true for cubic preferences, although we have not yet proved this result. We do know, however, that $\mathcal{C}(G_n)$ always contains at least one completely nonseparable preference matrix. This conclusion follows from recent work by Chew and Warner \cite{chewwarner1}, who proved that the consistency set of any Hamiltonian graph always contains a completely nonseparable preference matrix. Therefore, we can conclude that the stack and weave functions, while useful in generating a wide variety of cubic preferences, cannot generate all such preferences. On the other hand, since completely nonseparable preferences are prevalent, it is a potential benefit that the stack and weave functions avoid them altogether and instead generate preferences within the rarer subclasses of partially separable cubic preferences.

\subsection{Using Stack and Weave to Prove Admissibility}

We will now present a brief example to illustrate how the stack and weave functions can be used to prove the admissibility of higher-dimensional characters. Consider the following character for an election on $7$ questions: \[ C = \{\emptyset, \{7\}, \{2, 7\}, \{2, 3, 7\}, \{2, 3, 4, 6, 7\}, \{1, 2, 3, 4, 6, 7\}, \{1, 2, 3, 4, 5, 6, 7\} \} \] Since $C$ is nested (in the sense of Theorem \ref{nested}), it is possible that $C = \character(P)$ for some $P \in \mathcal{C}(G_7)$. We will suppose that such a matrix $P$ exists and work backwards to determine its form. Note that $Q_7 - \{5\} = \{1, 2, 3, 4, 6, 7\} \in C$. Therefore, Theorem \ref{weave-col} implies that $P = \weave_5^1(P')$ or $P = \weave_5^0(P')$ for some $P' \in \mathcal{C}(G_6)$. Moreover, Theorem \ref{weave-char} implies that \[ \character(P') = \{\emptyset, \{6\}, \{2, 6\}, \{2, 3, 6\}, \{2, 3, 4, 5, 6\}, \{1, 2, 3, 4, 5, 6\} \}. \]  Now note that $Q_6 - \{1\} = \{2,3,4,5,6\} \in \character(P')$. Applying Theorems \ref{weave-col} and \ref{weave-char} again, we can conclude that $P' = \weave_1^1(P'')$ or $P' = \weave_1^0(P'')$ for some $P'' \in \mathcal{C}(G_5)$ with \[ \character(P'') = \{ \emptyset, \{5\}, \{1, 5\}, \{1, 2, 5\}, \{1,2,3,4,5\} \}. \] Since $\{5\} \in \character(P'')$, Theorem \ref{stack-col} implies that $P'' = P_1 \stack_5^1 P_2$ or $P'' = P_1 \stack_5^0 P_2$ for some $P_1 \in \mathcal{C}(G_4)$ and some $P_2 \in \mathcal{S}(P_1)$. Finally, Theorem \ref{stack-char} implies that \[ \character(P_1) \cap \character(P_2) = \{ \emptyset, \{1\}, \{1, 2\}, \{1, 2, 3, 4\} \}. \]

Note that all of these steps are reversible. Therefore, if we can find $P_1$ and $P_2$ satisfying the above condition, we will have proved the admissibility of the original character $C$. It is clear from Table \ref{tableG4p} that there exists a matrix $P_1 \in \mathcal{C}(G_4)$ with $\character(P_1) = \{ \emptyset, \{1\}, \{1, 2\}, \{1, 2, 3, 4\} \}$. Form $P_2$ by taking bitwise complements of one or more of the columns of $P_1$ so that the first row of $P_2$ is equal to the last row of $P_1$. Then \[ \character(P_1) = \character(P_2) = \{ \emptyset, \{1\}, \{1, 2\}, \{1, 2, 3, 4\} \}, \] as desired. Theorems \ref{stack-char} and \ref{weave-char} now imply that if $P = \weave_5^1(\weave_1^1(P_1 \stack_5^1 P_2))$, then $\character(P) = C$. Thus, $C$ is admissible.

\section{Conclusion}\label{sec-conclusion}

In this paper, we introduced a new method of generating multidimensional binary preferences by associating them with Hamiltonian paths in graphs. Applying our method to the hypercube graph, we defined the class of cubic preferences and studied its properties. We showed how the algebraic structure of the hypercube influences the separability structure of the preference matrices it generates, and we proved that the character of any cubic preference matrix must be nested---a significant distinction from the characters of preferences generated by previous methods such as preseparable extensions. Finally, we defined two functions that can be used to construct cubic preferences with predictable characters, and these functions gave us broader insights into the types of characters that can be associated with cubic preferences.

Our work here stopped short of completely classifying cubic preferences, and further research is needed to describe in full generality the kinds of characters that can be associated with cubic preferences. In addition, some interesting combinatorial questions remain open. For example, it seems likely---but has not yet been proved---that for cubic preferences, the probability of complete nonseparability approaches $1$ as $n\rightarrow \infty$.

Beyond cubic preferences, the graph theoretic approach to preference generation provides a new method for addressing the character admissibility problem by allowing for the systematic construction of preference matrices with certain given characters. Different graphs will necessarily lead to different classes of preferences, which could then be studied using the methods developed here.

This work, and the character admissibility problem in general, has potential applications to experimental work involving election simulation. In order to accurately assess the impact of nonseparability on the outcomes of referendum elections, and to test methods for alleviating the separability problem, it is necessary to generate electorates with a diverse range of interdependence structures. Therefore, cubic preferences, and others generated via Hamiltonian paths in graphs, could be incorporated into future simulations to test both the effects of nonseparability and the robustness of alternative voting schemes.

\section*{References}



\begin{thebibliography}{1}

	\bibitem{bradley05} Bradley, W.J., Hodge, J.K., \& Kilgour, D.M. (2005). Separable discrete preferences. \emph{Mathematical Social Sciences} 49, 335--353.

	\bibitem{sepproblem} Brams, S.J., Kilgour, D.M., \& Zwicker, W.S. (1997). Voting on referenda: The separability problem and possible solutions. \emph{Electoral Studies} 16(3), 359--377.

	\bibitem{chewwarner1} Chew, S., \& Warner, T.J. (2015). REU final report. Available at \texttt{http://gvsu.edu/s/0Bo}.




	\bibitem{hodgekrineslahr09} Hodge, J.K., Krines, M., \& Lahr, J. (2009). Preseparable extensions of multidimensional preferences. \emph{Order} 26(2), 125--147.

	\bibitem{hodgeterhaar08} Hodge, J.K., \& TerHaar, M. (2008). Classifying interdependence in multidimensional binary preferences. \emph{Mathematical Social Sciences} 55, 190--204.

	\bibitem{lacyniou} Lacy, D., \& Niou, E.M. (2000). A problem with referendums. \emph{Journal of Theoretical Politics} 12(1), 5--31.

	\bibitem{wolfram-ham}  Weisstein, E.W. Hamiltonian graph. \emph{MathWorld---A Wolfram Web Resource}. Available at \\ \texttt{http://mathworld.wolfram.com/HamiltonianGraph.html}.

	\bibitem{wolfram-gray} Weisstein, E.W. Gray code. \emph{MathWorld---A Wolfram Web Resource}. Available at \texttt{http://mathworld.wolfram.com/GrayCode.html}.


\end{thebibliography}
\end{document}